\documentclass[11pt,twoside]{article}

\usepackage{mathrsfs,amsfonts,amsmath,amssymb}
\usepackage{latexsym}
\newtheorem{satz}{Theorem}

\newtheorem{theorem}[satz]{Theorem}
\newtheorem{lemma}[satz]{Lemma}
\newtheorem{definition}[satz]{Definition}
\newtheorem{corollary}[satz]{Corollary}
\newtheorem{remark}[satz]{Remark}

\def\_phi{\varphi}

\def\a{\alpha}

\def\la{\lambda}

\def\t{\tilde}

\def\ov{\overline}

\def\C{{\mathbb C}}
\def\R{{\mathbb R}}
\def\E{\mathsf {E}}

\def\Z_N{{\mathbb Z}_N}
\def\Z{{\mathbb Z}}

\def\Gr{{\mathbf G}}

\def\D{{\mathbb D}}

\def\Spec{{\rm Spec\,}}

\def\oT{{\rm T}}

\def\c{\circ}
\def\D{\Delta}
\def\Cf{{\mathcal C}}

\author{Shkredov I.D.}
\title{ On sums of Szemer\'{e}di--Trotter sets
\footnote{
This work was supported by grant
Russian Scientific Foundation RSF 14--11--00433.}
}
\date{}
\begin{document}
\maketitle

\begin{center}
 Annotation.
\end{center}

{\it \small
    We prove new general
    results
    on sumsets of sets
    having
    Szemer\'{e}di--Trotter type.
    This  family includes
    convex sets, sets with small multiplicative doubling, images of sets under convex/concave maps and others.
}
\\

\section{Introduction}
\label{sec:introduction}

Let $A=\{a_1,\dots,a_n\},\, a_i<a_{i+1}$ be  a set of real numbers.
We say that $A$ is \emph{convex} if
$$a_{i+1}-a_i>a_i-a_{i-1}$$ for every $i=2,\dots,n-1.$
Lower bounds for sumsets/difference sets of convex sets were obtained in several papers, see
\cite{h},
\cite{g1}, \cite{gk},
\cite{ENR}, \cite{ik}, \cite{soly},  \cite{ss2}, \cite{Li}, \cite{schoen_E_3} and others.
For example, in \cite{Li} the following theorem was proved.

\begin{theorem}
    Let $A\subset \R$ be a convex set. Then
$$
    |A+A| \gg |A|^{\frac{14}{9}} \log^{-\frac{2}{9}} |A| \,.
$$
\label{t:convex_sumset}
\end{theorem}

In our paper we obtain a series of results on sumsets/difference sets of rather general families of sets,
including convex sets,
see Theorems \ref{t:main}, \ref{t:main_diff}
below.
In particular, it allows us to refine the result above.

\begin{theorem}
    Let $A \subset \R$ be a convex set.
    Then
$$
    |A+A| \gg |A|^{\frac{58}{37}} \log^{-\frac{20}{37}} |A| \,.
$$
\label{t:convex}
\end{theorem}

Moreover, our method gives a generalization of Theorem \ref{t:convex_sumset} for sumset of two {\it different} convex sets,
see Theorem \ref{t:main_diff} below.

In \cite{Li2} the authors prove a general statement on addition of a set and its image under a convex map
(the first result in the direction was obtained in \cite{ENR}).

\begin{theorem}
Let $f$ be any continuous, strictly convex or concave function on the reals, and $A,C\subset \R$ be any finite sets such that $|A| = |C|$.
Then
$$
    |f(A)+C|^{10} |A+A|^9 \gg |A|^{24} \log^{-2} |A| \,.
$$
In particular, choosing $C=f(A)$, we get
$$
    \max\{ |f(A)+f(A)|, |A+A| \} \gg |A|^{\frac{24}{19}} \log^{-\frac{2}{19}} |A| \,.
$$
Finally
$$
    |AA|^{10} |A+A|^9 \gg |A|^{24} \log^{-2} |A| \,.
$$
\label{t:LR-N}
\end{theorem}

We refine the result.

\begin{theorem}
Let $f$ be any continuous, strictly convex or concave function on the reals, and $A,C\subset \R$ be any finite sets such that $|A| = |C|$.
Then
$$
    |f(A)+C|^{42} |A+A|^{37} \gg |A|^{100} \log^{-20} |A| \,.
$$
In particular, choosing $C=f(A)$, we get
$$
    \max\{ |f(A)+f(A)|, |A+A| \} \gg |A|^{\frac{100}{79}} \log^{-\frac{20}{79}} |A| \,.
$$
Finally
$$
    |AA|^{42} |A+A|^{37} \gg |A|^{100} \log^{-20} |A| \,.
$$
\label{t:LR-N'}
\end{theorem}

Another applications can be found in the last section \ref{sec:proof}.

In the proof we use so--called the eigenvalues method, see e.g. \cite{s_mixed} and some observations from \cite{Sh_ineq}.


\section{Notation}
\label{sec:definitions}

Let $\Gr$ be an abelian group and $+$ be the group operation.
In the paper we use the same letter to denote a set $S\subseteq \Gr$
and its characteristic function $S:\Gr \rightarrow \{0,1\}.$
By $|S|$ denote the cardinality of $S$.

Let $f,g : \Gr \to \C$ be two functions with finite supports.
Put
\begin{equation}\label{f:convolutions}
    (f*g) (x) := \sum_{y\in \Gr} f(y) g(x-y) \quad \mbox{ and } \quad
        (f\circ g) (x) := \sum_{y\in \Gr} f(y) g(y+x)
\end{equation}

Let $A_1,\dots, A_k \subseteq \Gr$ be any sets.
Put
\begin{equation}\label{f:E_k_preliminalies_B}
    \E_k (A_1,\dots,A_k)=\sum_{x\in \Gr} (A_1 \c A_1)(x) \dots (A_k \c A_k)(x)
\end{equation}
be {\it the higher energy}  of $A_1,\dots, A_k$.
If $A_j=A$, $j=1,\dots, k$ we simply write $\E_k (A)$ instead of $\E_k (A,\dots, A).$
In the same way one can define $\E_k (A)$  for non--integer $k$.
In particular case $k=2$ we put $\E(A,B):= \E_2 (A,B)$ and $\E(A) = \E_2 (A)$.
The quantity $\E(A)$ is called {\it the additive energy} of a set, see e.g. \cite{TV}.
Similarly, we
define
\begin{equation}\label{f:E_k_preliminalies_B_func}
    \E_k (f_1,\dots,f_{k}) = \sum_x (f_1 \c f_1) (x) \dots (f_k \c f_k) (x) \,.
\end{equation}
Denote by
$
\Cf_{k+1} (f_1,\dots,f_{k+1}) (x_1,\dots, x_{k})
$
the function
$$
    \Cf_{k+1} (f_1,\dots,f_{k+1}) (x_1,\dots, x_{k}) = \sum_z f_1 (z) f_2 (z+x_1) \dots f_{k+1} (z+x_{k}) \,.
$$
Thus, $\Cf_2 (f_1,f_2) (x) = (f_1 \circ f_2) (x)$.
If $f_1=\dots=f_{k+1}=f$ then write
$\Cf_{k+1} (f) (x_1,\dots, x_{k})$ for $\Cf_{k+1} (f_1,\dots,f_{k+1}) (x_1,\dots, x_{k})$.
Note that
\begin{equation}\label{f:E_k_via_C}
    \sum_{x_1,\dots,x_k} \Cf^2_{k+1} (f_1,\dots,f_{k+1}) (x_1,\dots, x_{k})
        =
            \E_{k+1} (f_1,\dots,f_{k+1}) \,.
\end{equation}

Let $g : \Gr \to \C$ be a function, and $A\subseteq \Gr$ be a finite set.
By $\oT^{g}_A$ denote the matrix with indices in the set $A$
\begin{equation}\label{def:operator1}
    \oT^{g}_A (x,y) = g(x-y) A(x) A(y) \,.
\end{equation}
It is easy to see that $\oT^{g}_A$ is hermitian iff $\ov{g(-x)} = g(x)$.
The corresponding action of $\oT_A^g$ is
$$
    \langle \oT^{g}_A a, b \rangle = \sum_z g(z) (\ov{b} \c a) (z) \,.
$$
for any functions $a,b : A \to \C$.
In the case $\ov{g(-x)} = g(x)$ by $\Spec (\oT^{g}_A)$ we denote the spectrum of the operator $\oT^{g}_A$
\begin{equation}\label{f:Spec_ordering}
    \Spec (\oT^{g}_A) = \{ \mu_1 \ge \mu_2 \ge \dots \ge \mu_{|A|} \} \,.
\end{equation}
Write  $\{ f \}_{\a}$, $\a\in [|A|]$ for the corresponding eigenfunctions.
We call
$\mu_1$ as
the main eigenvalue and $f_1$ as the main function.

In the asymmetric case
let $g : \Gr \to \C$ be a function, and $A,B\subseteq \Gr$ be two finite sets.
Suppose that $|B| \le |A|$.
By $\oT^{g}_{A,B}$ denote the rectangular matrix
\begin{equation}\label{def:operator1'}
    \oT^{g}_{A,B} (x,y) = g(x-y) A(x) B(y) \,,
\end{equation}
and by $\t{\oT}^{g}_{A,B} (x,y)$ denote the another rectangular  matrix
\begin{equation}\label{def:operator2'}
    \t{\oT}^{g}_{A,B} (x,y) = g(x+y) A(x) B(y) \,.
\end{equation}
As in (\ref{f:Spec_ordering}), we arrange the singular values in order of magnitude
$$
    \la_1 (\oT^{g}_{A,B}) \ge \la_2 (\oT^{g}_{A,B}) \ge \dots \ge \la_{|B|} (\oT^{g}_{A,B}) \ge 0 \,,
$$
$$
    \oT^{g}_{A,B} (x,y) = \sum_{j=1}^{|B|} \la_j u_j (x) v_j (y)
$$
and similar for $\t{\oT}^{g}_{A,B}$.
Here $u_j(x)$, $v_j (y)$ are singularfunctions of the
operators.
General theory of such operators was developed in \cite{s_mixed}.

All logarithms are base $2.$ Signs $\ll$ and $\gg$ are the usual Vinogradov's symbols.

\section{The main definition}
\label{sec:preliminaries}

We begin with a rather general definition of families of sets which are usually  obtained by Szemer\'{e}di--Trotter's theorem, see \cite{TV}.

\begin{definition}
    A set $A\subset \R$
    has
    {\bf SzT--type}
    (in other words $A$ is  called {\bf Szemer\'{e}di--Trotter set})
    with parameter $\a \ge 1$
    if for any set $B\subset \R$ and an arbitrary $\tau \ge 1$ one has
\begin{equation}\label{f:SzT-type}
    |\{ x\in A+B ~:~ (A * B)(x) \ge \tau \}| \ll c (A) |B|^\a \cdot \tau^{-3} \,,
\end{equation}
    where $c (A)>0$ is a constant depends on the set $A$ only.
    We define the quantity $c (A) |B|^\a$ as $c(A,B)$.
\label{def:SzT-type}
\end{definition}

From the definition
one can
see that if $A$ has SzT--type
then $(-A)$ has the same SzT--type
with the same parameters $\a$ and $c(A)$.

\begin{remark}
    We put parameter $\a\ge 1$ because otherwise there is no any  SzT--type set.
    Indeed, take $B=C-A$, where set $C$ will be chosen later.
    One has $(A*B) (x) \ge |A|$ for any $x\in C$.
    Then by (\ref{f:SzT-type}), we obtain
$$
    |A|^3 |C| \ll c(A) |A+C|^\a \le c(A) |A|^\a |C|^\a \,.
$$
    Taking $|C|$ sufficiently large and having the set $A$ is fixed, we see that $\a\ge 1$.
\end{remark}

{\bf Examples.} Let us give some examples of SzT--type sets with parameter $\a=2$. \\
%
%
$1)~$ If $A\subset \R$ is a convex set then $A$ has SzT--type with $c(A) = |A|$, see \cite{ss2}.\\
$2)~$ Let $f$ be a strictly convex/concave function. Then $f(A)$ has SzT--type with
$c(A) = q(A)$,
where
\begin{equation}\label{f:q(A)}
    q (A) := \min_{C}
        \frac{|A+C|^2}{|C|}
        \,,
\end{equation}
and $A$ has SzT--type with $c(A) = q(f(A))$,
see \cite{SS1}, \cite{Li2}.\\
$3)~$ Let $|AA| \le M|A|$. Then $A$ has SzT--type with $c(A) = M^2 |A|$.
This is a particular case of the family from (2).
Indeed, take $f(x) = \log x$, and apply (\ref{f:q(A)}) with
$C=\log (A \cap \R^+)$ or $C=\log |(A \cap \R^-)|$.\\
$4)~$ Let $A \subset \R^{+}$, and $a \in \R \setminus \{0 \}$.
Then $\log A$ has SzT--type with $c(A) = q'(A)$, where
$$
    q' (A) := \min_{C}
        \frac{|(A+a)C|^2}{|C|}
        \,,
$$
see \cite{J_RN}, \cite{B_RN_S}.

There are another families of SzT--type sets, for example see a family of complex sets in \cite{KR}.

\bigskip

Using definition \ref{def:SzT-type} and  easy calculations,
one can obtain
upper bounds for some simple characteristics of SzT--type sets see, e.g.
papers \cite{SS1}, \cite{ss2}, \cite{Li}, \cite{Li2}.
It is more convenient do not use parameter $\a$ in the statements.

\begin{lemma}
    Let $A$ be a SzT--type set.
    Then
$$
    \E_3 (A) \ll c(A,A) \log |A| \,, \quad \quad \E^3 (A) \ll \E^{2}_{3/2} (A) c^{} (A,A) \,,
$$
    and for any $B$ one has
$$
    \E(A,B) \ll (c(A,B) |A| |B|)^{1/2} \,.
$$
\label{l:SzT}
\end{lemma}

We need in one more technical lemma.

\begin{lemma}
    Let $A$, $A_*$ be a SzT--type set with the same parameter $\a>1$.
    Then
$$
    \E^{2\a-1} (A_*,A) \ll_{(\a-1)^{-1}}
            \left( \sum_x (A_* \c A_*)^{1/2} (x) (A\c A) (x) \right)^{2\a-2}
        \cdot
            c^{1/3} (A) c^{\a/3} (A_*) |A|^{2/3} |A_*|^{\a^2/3} \,.
$$
\label{l:SzT_1}
\end{lemma}
\begin{proof}
Put $c_* = c(A_*)$, $c=c(A)$, $a=|A|$, $a_* = |A_*|$.
Splitting the sum,
we get with some inaccuracy
\begin{equation}\label{tmp:20.10.2014_1}
    \E(A_*,A) \ll \tau^{1/2}  \left( \sum_x (A_* \c A_*)^{1/2} (x) (A\c A) (x) \right)
        +
            \tau \sum_x S_\tau (x) (A \c A) (x)
                =
                    \tau^{1/2} \omega_1 + \tau \omega_2 \,,
\end{equation}
where $S_{\tau} = \{ x ~:~ (A_* \c A_*) (x) \ge \tau \}$.
Because of $A_*$
is Szemer\'{e}di--Trotter set,
we have  $|S_{\tau}| \ll c_* (a_*)^\a \tau^{-3}$.
On the other hand,  $A$ is also Szemer\'{e}di--Trotter set, so
\begin{equation}\label{tmp:20.10.2014_2}
    \sum_x S_\tau (x) (A \c A) (x) = \sum_{x\in A} (S_\tau * A) (x)
        \ll
            c^{1/3} |S_\tau|^{\a/3} a^{2/3}
                \ll
                    c^{1/3} a^{2/3} c_*^{\a/3} (a_*)^{\a^2/3} \tau^{-\a} \,.
\end{equation}
Combining (\ref{tmp:20.10.2014_1}) and (\ref{tmp:20.10.2014_2}), we obtain
$$
    \E(A_*,A) \ll_{(\a-1)^{-1}} \tau^{1/2}  \omega_1
        +
            \tau^{1-\a} c^{1/3} a^{2/3} c_*^{\a/3} (a_*)^{\a^2/3} \,.
$$
An optimal choice of parameter $\tau$ is
$\tau^{1/2} = \omega^{-1/(2\a-1)}_1 c^{1/3(2\a-1)} c^{\a/3(2\a-1)}_* a^{2/3(2\a-1)} a^{\a^2/3(2\a-1)}_*$.
Thus
$$
    \E^{2\a-1} (A_*,A) \ll_{(\a-1)^{-1}}
            \left( \sum_x (A_* \c A_*)^{1/2} (x) (A\c A) (x) \right)^{2\a-2}
        \cdot
            c^{1/3}  c^{\a/3}_*  a^{2/3} a_*^{\a^2/3}
$$
as required.
$\hfill\Box$
\end{proof}

\section{The proof of the main result}
\label{sec:proof}

We begin with a lemma from \cite{s_mixed}.

\begin{lemma}
    Let $A\subseteq \Gr$ be a set and $g$ be a nonnegative function on $\Gr$.
    Suppose that $f_1$ is the main eigenfunction of $\oT^g_A$ or $\t{\oT}^g_A$,
    and $\mu_1$ is the correspondent eigenvalue.
    Then
$$
    \langle \oT^{A\c A}_A f_1 f_1 \rangle \ge \frac{\mu^3_1}{\| g\|_2^2 \cdot \| g\|_{\infty}} \,.
$$
\label{l:action_g}
\end{lemma}

A particular case $\a=2$ of
the next lemma is contained inside Theorem 8 of paper \cite{Sh_ineq}.
We give the proof for the sake of completeness.

\begin{lemma}
    Let $A$ be a SzT--type set
    and let $\D \ge 1$ be a real number.
    Suppose that
    $$B \subseteq \{ x ~:~ (A\c A) (x) \ge \Delta \}
        \quad
    \mbox{ or }
        \quad
    B \subseteq \{ x ~:~ (A* A) (x) \ge \Delta \} \,.$$
    Then
\begin{equation}\label{f:E_3(A,A,B)}
    \E_3(A,A,B)
        \ll
            \D^{-\frac{4}{3\a-1}} c(A)^{\frac{5\a+1}{2(3\a-1)}} |A|^{\frac{2\a^2+5\a-1}{2(3\a-1)}} |B|^{\frac{3(\a^2-1)}{2(3\a-1)}} \log |A| \,.
\end{equation}
\label{l:E_3(A,A,B)}
\end{lemma}
\begin{proof}
    Put $a=|A|$, $L = \log a$ and $c=c(A)$.
    By the pigeonhole principle there is a set $Q$ such that
$$
    \sigma := \E_3(A,A,B)
        \ll
            L \sum_{x\in Q} (A\c A)^2 (x) (B\c B) (x) \,,
$$
    and the values of the convolution $(A\c A) (x)$ differ at most twice on $Q$.
    Denote  by $q$ the maximum of $(A\c A) (x)$ on $Q$.
    Because of $A$ is SzT--type set, we have   $|Q| \ll c a^\a q^{-3}$.
    Using Lemma \ref{l:SzT}, we obtain
\begin{equation}\label{tmp:14.10.2014_1}
    \sigma \ll L q \E (A,B) \ll L q (c(A) a |B|^{\a+1})^{1/2} \,.
\end{equation}
    On the other hand, by the definition of the set $B$
    and the Cauchy--Schwarz inequality one has
\begin{equation}\label{tmp:14.10.2014_2}
    \sigma \ll L q^2 \D^{-1} \E(Q,B,A,A)
        \ll  L q^2 \D^{-1} \E^{1/2} (Q,A) \E^{1/2} (B,A) \,.
\end{equation}
    Combining the last two bounds and the upper bound for size of $Q$, we have
$$
    \sigma
        \ll
            L q \E^{1/2} (A,B) (\E^{1/2} (A,B) + q \D^{-1} c^{1/4} a^{1/4} |Q|^{(\a+1)/4})
                \ll
$$
$$
                \ll
                     L q \E^{1/2} (A,B) (\E^{1/2} (A,B) + q^{-(3\a-1)/4} \D^{-1} c^{(\a+2)/4} a^{(\a^2+\a+1)/4} ) \,.
$$
The optimal choice of $q$ is $q=\E^{-2/(3\a-1)} (A,B) \D^{-4/(3\a-1)} c^{(\a+2)/(3\a-1)} a^{(\a^2+\a+1)/(3\a-1)}$.
Here we have used the fact that $\a > 1/3$.
Substituting $q$ into the last formula and using Lemma \ref{l:SzT} again, we get
$$
    \sigma
        \ll
            L q \E (A,B)
                \ll
                    L \E^{3(\a-1)/(3\a-1)} (A,B) \cdot \D^{-4/(3\a-1)} c^{(\a+2)/(3\a-1)} a^{(\a^2+\a+1)/(3\a-1)}
                        =
$$
$$
    =
         L \D^{-\frac{4}{3\a-1}} |B|^{\frac{3(\a^2-1)}{6\a-2}} c^{\frac{5\a+1}{6\a-2}} a^{\frac{2\a^2+5\a-1}{6\a-2}}
$$
as required.
$\hfill\Box$
\end{proof}

\bigskip

Let us formulate the main result of the paper.

\begin{theorem}
    Suppose that $A\subset \R$ has SzT--type with parameter $\a$.
    Then
\begin{equation}\label{f:p_main}
    |A+A| \gg  c(A)^{\frac{1-11\a}{3\a^2+12\a+1}} |A|^{\frac{-8\a^2+57\a-3}{3\a^2+12\a+1}}
        \cdot
            (\log |A|)^{-\frac{4(3\a-1)}{3\a^2+12\a+1}} \,.
\end{equation}
    In particular, for $\a=2$ one has
\begin{equation}\label{f:p_main_1}
    |A+A| \gg  c(A)^{-\frac{21}{37}} |A|^{\frac{79}{37}}
        \cdot
            (\log |A|)^{-\frac{20}{37}} \,.
\end{equation}
\label{t:main}
\end{theorem}
\begin{proof}
Let $S=A+A$, $|S| = d$, $a=|A|$.
Let also $L=\log a$, $c=c(A)$.
We have
\begin{equation}\label{tmp:14.10.2014_3}
    |A|^2 = \sum_{x,y} A(x) A(y) S(x+y) \le 2\sum_{z\in S_1} (A*A) (z) \,,
\end{equation}
where $S_1 = \{ z\in S ~:~ (A*A) (z) \ge 2^{-1} a^2 d^{-1} \}$.
Denote by $f_j$, $\mu_j$ the eigenfunctions and eigenvalues of hermitian operator
$\t{\oT}^{S_1}_A$.
From (\ref{tmp:14.10.2014_3}) it follows that
$\mu_1 \ge a/2$.
Applying Lemma \ref{l:action_g}, we see that
\begin{equation}\label{tmp:14.10.2014_4}
    \langle \oT^{A\c A}_A f_1, f_1 \rangle \ge \mu^3_1 (\t{\oT}^{S_1}_A) d^{-1} \ge 2^{-3} a^3 d^{-1} \,.
\end{equation}
Further, by nonnegativity of the operator $\oT^{A\c A}_A$ as well as
inequality (\ref{tmp:14.10.2014_4})
and the lower bound for $\mu_1$, we get
\begin{equation}\label{tmp:14.10.2014_5}
    \sigma:= \sum_{x,y,z} \oT^{A\c A}_A (x,y) \t{\oT}^{S_1}_A (x,z) \t{\oT}^{S_1}_A (y,z)
        =
            \sum_{j=1}^{a} \mu^2_j \langle \oT^{A\c A}_A f_j, f_j \rangle
                \ge
                    \mu^2_1
                        \langle \oT^{A\c A}_A f_1, f_1 \rangle
                            \ge
                                2^{-5} a^5 d^{-1} \,.
\end{equation}
On the other hand
\begin{equation}\label{tmp:14.10.2014_6}
    \sigma
        =
            \sum_{x,y,z \in A} (A\c A)(x-y) S_1 (x+z) S_1 (y+z)
                =
                    \sum_{\a,\beta} S_1 (\a) S_1 (\beta) (A\c A) (\a-\beta) \Cf_3 (-A,A,A) (\a,\beta) \,.
\end{equation}
Combining (\ref{tmp:14.10.2014_5}), (\ref{tmp:14.10.2014_6}) and using (\ref{f:E_k_via_C}), we obtain
by the Cauchy--Schwarz inequality that
\begin{equation}\label{tmp:14.10.2014_7}
    a^{10} d^{-2}
        \ll
            \E_3 (A) \E_3 (A,A,S_1) \,.
\end{equation}
Applying the first formula of Lemma \ref{l:SzT} to estimate the quantity  $\E_3 (A)$ and Lemma \ref{l:E_3(A,A,B)}
to estimate $\E_3 (A,A,S_1)$, we have
$$
    a^{10} d^{-2}
        \ll
            L^2 a^\a c  \cdot
                \D^{-\frac{4}{3\a-1}} c^{\frac{5\a+1}{2(3\a-1)}} a^{\frac{2\a^2+5\a-1}{2(3\a-1)}} d^{\frac{3(\a^2-1)}{2(3\a-1)}} \,,
$$
where $\D = 2^{-1} a^2 d^{-1}$.
After some calculations, we get
$$
    d \gg L^{-\frac{4(3\a-1)}{3\a^2+12\a+1}} c^{\frac{1-11\a}{3\a^2+12\a+1}} a^{\frac{-8\a^2+57\a-3}{3\a^2+12\a+1}}
$$
as required.
$\hfill\Box$
\end{proof}

\bigskip

{\bf Proof of Theorems \ref{t:convex}, \ref{t:LR-N'}.}
To obtain Theorem \ref{t:convex} just recall that $\a=2$, $c(A) = |A|$ for convex sets.
Remembering the definition of $q(f(A))$ from (\ref{f:q(A)}), we have
$c(A) = q(f(A)) \le |f(A)+C|^2 |C|^{-1}$.
After that applying the main Theorem \ref{t:main}, we get Theorem \ref{t:LR-N'}.
$\hfill\Box$

\begin{remark}
    Of course, one can replace the condition $|A| = |C|$ in Theorems \ref{t:LR-N}, \ref{t:LR-N'}
    to $c_1 |A| \le |C| \le c_2 |A|$, where $c_1, c_2>0$ are any absolute constants.
    Certainly,
    signs
    $\ll,\gg$ should be
    changed
    by
    $\ll_{c_1, c_2}, \gg_{c_1, c_2}$ in the case.
    Even more, it is possible to prove the results for sets $A$ and $C$,  having incomparable sizes.
    We do not make such calculations here (and also below), note only that because in Theorem \ref{t:LR-N'}, we have
    $c(A) \le |f(A)+C|^2 |C|^{-1}$ this  implies
$$
    |f(A)+C|^{42} |A+A|^{37} \gg |A|^{79} |C|^{21} \log^{-20} |A| \,.
$$
\end{remark}

\bigskip

We conclude the section proving a result which generalize, in particular, Theorem 1.3 from \cite{Li2} as well as
the results on sumsets/difference sets of convex sets from \cite{ss2}.
The arguments are in the spirit of Theorem \ref{t:main}.
We need in a lemma from \cite{s_mixed}.

\begin{lemma}
    Let $A,B\subseteq \Gr$ be finite sets, $|B| \le |A|$,
    $D,S\subseteq \Gr$ be two sets such that
    $A-B \subseteq D$, $A+B \subseteq S$.
    Then the main eigenvalues and singularfunctions
    of the operators $\oT^{D}_{A,B}$, $\t{\oT}^{S}_{A,B}$
    equal
    $\la_1 = (|A||B|)^{1/2}$,
    and
    $$
        v_1 (y) = B(y)/|B|^{1/2}\,,
            \quad \mbox{ and } \quad
        u_1 (x) = A(x)/|A|^{1/2}
        \,.
    $$
    All other singular values equal zero.
\label{l:eigenvalues_D,S'}
\end{lemma}

Using lemma above,
we prove our second main result,
although one can use a more elementary approach as in \cite{ss2}.

\begin{theorem}
    Suppose that $A,A_*\subset \R$ have SzT--type with the same parameter $\a$.
    Then
$$
    |A \pm A_*| \gg
$$
$$
        \min\{
        c (A_*)^{-\frac{2}{3(7+\a)}} c(A)^{-\frac{13}{3(7+\a)}} |A_*|^{\frac{2(24-\a)}{3(7+\a)}} |A|^{\frac{33-10\a}{3(7+\a)}},\,
        c (A)^{-\frac{2}{3(7+\a)}} c(A_*)^{-\frac{13}{3(7+\a)}} |A|^{\frac{2(24-\a)}{3(7+\a)}} |A_*|^{\frac{33-10\a}{3(7+\a)}}
            \}
$$
\begin{equation}\label{f:p_main_diff}
            \times
            (\log (|A| |A_*|))^{-\frac{2}{7+\a}}
            \,,
\end{equation}
and for $\a>1$
$$
    |A \pm A_*| \gg_{(\a-1)^{-1}} \,
        c (A)^{-\frac{4\a-2}{3(\a^2+4\a-3)}} c(A_*)^{-\frac{7\a-5}{3(\a^2+4\a-3)}} |A|^{\frac{28\a-4\a^2-16}{3(\a^2+4\a-3)}} |A_*|^{\frac{35\a-4\a^2-21}{3(\a^2+4\a-3)}}
$$
\begin{equation}\label{f:p_main_diff_new}
        \times
            (\log (|A| |A_*|))^{-\frac{2(\a-1)}{\a^2+4\a-3}} \,.
\end{equation}
    In particular, for $\a=2$ one has
$$
    |A \pm A_*| \gg
        \max\{
        c (A_*)^{-\frac{1}{3}} c(A)^{-\frac{2}{9}} |A_*|^{\frac{11}{9}} |A|^{\frac{8}{9}},
        c (A)^{-\frac{1}{3}} c(A_*)^{-\frac{2}{9}} |A|^{\frac{11}{9}} |A_*|^{\frac{8}{9}} \,,
$$
$$
        \min\{
        c (A_*)^{-\frac{2}{27}} c(A)^{-\frac{13}{27}} |A_*|^{\frac{44}{27}} |A|^{\frac{13}{27}},
        c (A)^{-\frac{2}{27}} c(A_*)^{-\frac{13}{27}} |A|^{\frac{44}{27}} |A_*|^{\frac{13}{27}} \} \}
            \times
$$
\begin{equation}\label{f:p_main_diff_1}
    \times
        (\log (|A| |A_*|))^{-\frac{2}{9}} \,.
\end{equation}
    Finally
\begin{equation}\label{f:p_main_diff_2}
    |A\pm A_*|^{\frac{\a+1}{2}} |A-A|
        \gg
            |A|^{\frac{33-4\a}{6}} |A_*|^{\frac{6-\a}{3}} c^{-7/6} (A) c^{-1/3} (A_*) \log^{-1} (|A| |A_*|) \,.
\end{equation}
\label{t:main_diff}
\end{theorem}
\begin{proof}
Because
SzT--types of the sets $A$ and $(-A)$ are coincide  it is sufficient to prove the result for sums.
Let $S=A+A_*$, $|S| = d$, $a=|A|$, $a_* = |A_*|$.
Let also $L=\log (a a_*)$, $c=c(A)$, $c_* = c (A_*)$.
By the Cauchy--Schwarz inequality, we have
\begin{equation}\label{tmp:14.10.2014_3'}
    a a_* = \sum_{x,y} A(x) A_* (y) S(x+y) \le d^{1/2} \E^{1/4} (A) \E^{1/4} (A_*) \,.
\end{equation}
Let us begin with (\ref{f:p_main_diff}).
One can assume that $\E(A) \ge d^{-1} a^2 a_*^2$, the opposite case is similar.
By Lemma \ref{l:SzT}, we get
\begin{equation}\label{tmp:15.10.2014_1}
    \frac{\E_{3/2} (A)}{a} \ge d^{-3/2} c^{-1/2} a^{2-\a/2} a_*^3 \,.
\end{equation}
Denote by $f_j$, $\mu_j$ the eigenfunctions and the eigenvalues of hermitian nonnegative operator
$$
    (\t{\oT}^{S}_{A,A_*} (\t{\oT}^{S}_{A,A_*})^* ) (x,y) = \Cf_3 (A_*,S,S) (x,y) A(x) A(y) \,.
$$
By Lemma \ref{l:eigenvalues_D,S'} we know that $f_1 (x) = A(x) / a^{1/2}$ and $\mu_1 = a a_*$.
Using the lemma again
as well as
bound (\ref{tmp:15.10.2014_1}), we obtain
$$
    \sigma:= \sum_{x,y \in A} \oT^{(A\c A)^{1/2}}_A (x,y) \Cf_3 (A_*,S,S) (x,y)
        =
            \sum_{j=1}^{a} \mu_j \langle \oT^{(A\c A)^{1/2}}_A f_j, f_j \rangle
                =
$$
\begin{equation}\label{tmp:14.10.2014_5'}
                =
                    a^{-1} \mu_1
                        \langle \oT^{(A\c A)^{1/2}}_A A, A \rangle
                            \ge
                                 d^{-3/2} c^{-1/2} a^{3-\a/2} a_*^4 \,.
\end{equation}
On the other hand, we have as in (\ref{tmp:14.10.2014_6}), (\ref{tmp:14.10.2014_7}) that
\begin{equation}\label{tmp:15.10.2014_2}
    \sigma^2 \le \E_3 (A_*,A,A) \E(A,S) \le \E_3 (A_*,A,A) (c a d^{\a+1})^{1/2} \,.
\end{equation}
Using calculations similar to Lemma \ref{l:SzT}, one can show that
\begin{equation}\label{tmp:15.10.2014_3}
    \E_3 (A_*,A,A)
        \ll
            (c_* c^2)^{1/3} (a_* a^2)^{\a/3} L \,.
\end{equation}
Substituting (\ref{tmp:15.10.2014_3}) into (\ref{tmp:15.10.2014_2})
and combining the result with (\ref{tmp:14.10.2014_5'}), we obtain
$$
    d^{-3} c^{-1} a^{6-\a} a_*^8
        \le
            (c a d^{\a+1})^{1/2} (c_* c^2)^{1/3} (a_* a^2)^{\a/3} L \,.
$$
After some calculations, we have
$$
    d \gg L^{-\frac{2}{7+\a}} c_*^{-\frac{2}{3(7+\a)}} c^{-\frac{13}{3(7+\a)}} a_*^{\frac{2(24-\a)}{3(7+\a)}} a^{\frac{33-10\a}{3(7+\a)}}
$$
as required.

To prove (\ref{f:p_main_diff_new}), returning to (\ref{tmp:14.10.2014_3'})
and applying Lemma \ref{l:SzT_1}, we obtain
\begin{equation*}\label{tmp:20.10.2014_3}
    ( a^2 a_*^2 d^{-1} )^{2\a-1}
        \le
            \E^{2\a-1} (A_*, A)
                \ll_{(\a-1)^{-1}}
                                \left( \sum_x (A_* \c A_*)^{1/2} (x) (A\c A) (x) \right)^{2\a-2}
        \cdot
            c^{1/3}  c^{\a/3}_*  a^{2/3} a_*^{\a^2/3} \,.
\end{equation*}
Thus
\begin{equation}\label{tmp:20.10.2014_3}
    a^{-1} \langle \oT^{(A_* \c A_*)^{1/2}}_A A, A \rangle
        \gg_{(\a-1)^{-1}}
            a^{\frac{3\a-1}{3(\a-1)}} a_*^{\frac{12\a-6-\a^2}{6(\a-1)}} d^{\frac{1-2\a}{2(\a-1)}} c^{-\frac{1}{6(\a-1)}} c_*^{-\frac{\a}{6(\a-1)}} \,.
\end{equation}
After that use previous arguments replacing $\oT^{(A \c A)^{1/2}}_A$
onto $\oT^{(A_* \c A_*)^{1/2}}_A$.
By (\ref{tmp:20.10.2014_3}), we have
$$
    a^{\frac{12\a-8}{3(\a-1)}} a_*^{\frac{18\a-12-\a^2}{3(\a-1)}} d^{\frac{1-2\a}{\a-1}} c^{-\frac{1}{3(\a-1)}} c_*^{-\frac{\a}{3(\a-1)}}
        \le
    \left( \mu_1 a^{-1} \langle \oT^{(A_* \c A_*)^{1/2}}_A A, A \rangle \right)^2
        \le
    \E_3 (A_*,A,A) \E(A_*,S)
        \le
$$
$$
        \le
            (c_* a_* d^{\a+1})^{1/2} (c_* c^2)^{1/3} (a_* a^2)^{\a/3} L \,.
$$
After some calculations, we get the required bound.

Finally,
to get
(\ref{f:p_main_diff_2}) just apply the arguments above to get
$$
    a_*^2 \E^2_{3/2}
        \le (c a d^{\a+1})^{1/2} (c_* c^2)^{1/3} (a_* a^2)^{\a/3} L
$$
and use the lower bound $\E^2_{3/2} (A) \ge a^6 / |A-A|$.
Of course, one can replace $A$ to $A_*$ in (\ref{f:p_main_diff_2}) and vice versa.
This completes the proof.
$\hfill\Box$
\end{proof}

\bigskip

Theorem above gives us a consequence to sumsets/difference sets for convex sets.

\begin{corollary}
    Let $A,A_*\subset \R$ be two convex sets.
    Then
$$
    |A \pm A_*| \gg
    \max\{ |A|^{\frac{8}{9}} |A_*|^{\frac{2}{3}},
           |A_*|^{\frac{8}{9}} |A|^{\frac{2}{3}}
        \}
            \cdot \log^{-\frac{2}{9}} (|A||A_*|) \,.
$$
\end{corollary}

In another corollary we
obtain
Theorem 1.3 from \cite{Li2}
as well as Corollary 1.4 from the paper.
These results can be considered as theorems on lower bounds for sums of SzT--type sets of special form.


\begin{corollary}
We have
\begin{equation}\label{tmp:19.10.2014_1}
    |A+f(A)| \gg \frac{|A|^{24/19}}{(\log |A|)^{2/19}}
\end{equation}
for any strictly convex or concave function $f$.
Further
\begin{equation}\label{tmp:19.10.2014_2}
    |AA|^6 |A-A|^5 \gg \frac{|A|^{14}}{\log^2 |A|} \,.
\end{equation}
In particular
$$
    \max\{ |AA|, |A-A| \} \gg |A|^{14/11} \log^{-2/11} |A| \,.
$$
\end{corollary}
\begin{proof}
Indeed, to obtain (\ref{tmp:19.10.2014_1}) just  apply (\ref{f:p_main_diff_1}) with $A=A$, $A_* = f(A)$ and
$c(A) = c_* (A) = |f(A)+A|^2 |A|^{-1}$.
To get (\ref{tmp:19.10.2014_2}), we use (\ref{f:p_main_diff_2}) with $A=A_*$ having
$$
    |A-A|^{5/2} \gg |A|^{11/2} c^{-3/2} (A) \log^{-1} |A| \,.
$$
After that recall
$c(A) = M^2 |A|$ with $M=|AA|/|A|$.
This concludes the proof.
$\hfill\Box$
\end{proof}

\bigskip

\noindent{I.D.~Shkredov\\
Steklov Mathematical Institute,\\
ul. Gubkina, 8, Moscow, Russia, 119991}
\\
and
\\
IITP RAS,  \\
Bolshoy Karetny per. 19, Moscow, Russia, 127994\\
{\tt ilya.shkredov@gmail.com}


\begin{thebibliography}{99}


\bibitem{g1}
{\sc M.~Z.~Garaev,} {\em On lower bounds for $L_1-$norm of exponential sums,} Mathematical Notes 68 (2000), 713--720.


\bibitem{gk}
{\sc M.~Z.~Garaev, K-L.~Kueh,} {\em On cardinality of sumsets,} J.
Aust. Math. Soc. 78 (2005), 221--224.


\bibitem{ENR}
{\sc G. Elekes, M. Nathanson, I. Ruzsa, }
{\em Convexity and sumsets, }
Journal of Number Theory, 83:194--201, 1999.


\bibitem{h}
{\sc N.~Hegyv\'ari,}
{\em On consecutive sums in sequences,} Acta Math. Acad. Sci. Hungar. 48 (1986), 193--200.


\bibitem{ik}
{\sc A.~Iosevich, V.~S.~Konyagin, M.~Rudnev, V.~Ten,} {\em On combinatorial complexity of convex sequences,} Discrete Comput. Geom. 35 (2006), 143--158.



\bibitem{J_RN}
{\sc T.~G.~F. Jones, O.~Roche--Newton, }
{\em Improved bounds on the set $A(A+1)$, }
J. Combin. Theory Ser. A 120 (2013), no. 3, 515--526.


\bibitem{KR}
{\sc S.V. Konyagin, M. Rudnev, }
{\em On new sum-product type estimates, }
SIAM J. Discrete Math. {\bf 27} (2013), no. 2, 973--990.


\bibitem{Li}
{\sc L.~Li, } {\em On a theorem of Schoen and Shkredov on sumsets of convex sets, }
arXiv:1108.4382v1 [math.CO].


\bibitem{Li2}
{\sc L.~Li, O.~Roche--Newton, } {\em Convexity and a sum--product type estimate, }
Acta Arith. 156 (2012), no. 3, 247--255.


\bibitem{B_RN_S}
{\sc B.~Murphy, O.~Roche--Newton, I.~D.~Shkredov, }
{\em Variations on the sum--product problem, }
arXiv:1312.6438v2 [math.CO] 8 Jan 2014.




\bibitem{schoen_E_3}
{\sc T. Schoen, }
{\em On convolutions of convex sets and related problems, }
preprint.


\bibitem{SS1}
{\sc T. Schoen, I.D. Shkredov, }
{\em Higher moments of convolutions, }
J. Number Theory 133 (2013), no. 5, 1693--1737.


\bibitem{ss2}
{\sc T.~Schoen, I.~D.~Shkredov,} {\em On sumsets of convex sets, }
Comb. Probab. Comput. {\bf 20} (2011), 793--798.


\bibitem{Sh_ineq}
{\sc I.D. Shkredov, }
{\em Some new inequalities in additive combinatorics, }
Moscow J. Combin. Number Theory 3 (2013), 237--288.


\bibitem{s_mixed}
{\sc I.D. Shkredov, }
{\em Some new results on higher energies, }
Transactions of MMS, 74:1 (2013), 35--73.








\bibitem{soly}
{\sc J.~Solymosi,} {\em Sumas contra productos,}
Gaceta de la Real Sociedad Matematica Espanola, ISSN 1138--8927,  12 (2009).



\bibitem{TV}
{\sc T. Tao, V. Vu, }
{\em Additive Combinatorics, }
Cambridge University Press (2006).




\end{thebibliography}
\end{document}